\documentclass{amsart}
\usepackage[backref]{hyperref}
\usepackage{amssymb}
\usepackage{fancyhdr}		
\usepackage{color}

\usepackage[T1]{fontenc}
\usepackage[latin1]{inputenc}
\usepackage{url}	
\usepackage{graphicx}	
\usepackage{amssymb}
\usepackage{amsmath}
\usepackage{amsthm}
\usepackage[T1]{fontenc}
\usepackage{ae}

\newtheorem{theo}{Theorem}[section]
\newtheorem{cor}[theo]{Corollary }
\newtheorem{lemm}[theo]{Lemma}

\newtheorem{defi}{Definition}[section]

\newtheorem{prop}[theo]{Proposition}
\newtheorem{Remark}[theo]{Remark}

\def\C{\mathbb{C}}

\def\Eins{\mathsf{Eins}}

\def\PGL{{\sf{PGL}}}
\def\PSL{{\sf{PSL}}}
\def\GL{{\sf{GL}}}
\def\SL{{\sf{SL}}}

\def\SO{{\sf{SO}}}
\def\PSO{{\sf{PSO}}}

\def\Sp{{\sf{Sp}}}
\def\PSp{{\sf{PSp}}}

\def\det{{\sf{det}}}

\def\h{{\mathfrak{h}}}

 \def\QQ{{\mathsf{Q}}}
  \def\Ein{{\mathsf{Eins}}}

\title[]{ON HOMOGENEOUS HOLOMORPHIC CONFORMAL
STRUCTURES}

\date{\today}
\author{M. Belraouti}
\address{Mehdi Belraouti \newline
Facult\'e de Math\'ematiques,\\
USTHB, BP 32, El-Alia,\\
16111 Bab-Ezzouar, Alger (Algeria)}
\email{mbelraouti@usthb.dz}

\author{M. Deffaf}
\address{Mohamed Deffaf \newline
Facult\'e de Math\'ematiques,\\
USTHB, BP 32, El-Alia,\\
16111 Bab-Ezzouar, Alger (Algeria)}
\email{mdeffaf@usthb.dz}

\author{Y. Raffed} 

\address{Yazid raffed \newline
Facult\'e de Math\'ematiques,\\
USTHB, BP 32, El-Alia,\\
16111 Bab-Ezzouar, Alger (Algeria) }
\email{yazidsaid.raffed@usthb.edu.dz}

\author{A. Zeghib}
\address{Abdelghani Zeghib \newline
UMPA, ENS de Lyon, France }
\email{abdelghani.zeghib@ens-lyon.fr}
\begin{document}
\maketitle

\noindent{\bf Abstract.}
We study compact complex  manifolds $M$ admitting a conformal holomorphic Riemannian  structure invariant under the action of a complex semi-simple Lie group $G$. We prove that if  the group $G$ acts transitively and essentially, then $M$ is conformally flat.

\tableofcontents

\section{Introduction}
Throughout this paper, $M$  will denote a compact connected complex manifold of dimension $n$. A holomorphic Riemannian metric $g$ on $M$ is a holomorphic field of non-degenerate complex quadratic forms on $TM$. Locally, it can be written as $\sum g_{ij}(z)dz_{i}dz_{j}$, where $(g_{ij}(z))$ is an invertible symmetric complex matrix depending holomorphically on $z$. It is the complex analogue of a pseudo-Riemannian metric. Unlike the real case, there are only few  compact complex manifolds admitting a holomorphic Riemannian metric. A first natural example is given by the flat standard model $\sum dz^{2}_{i}$ on $\mathbb{C}^{n}$. Since this metric is invariant under translations, any complex torus  admits  a holomorphic Riemannian metric. Actually, up to finite cover, complex torus are the  only compact  Kähler manifolds admitting such structure (see \cite{Kobay}).

Consider a cover $\left\{U_{i}\right\}$ of $M$, along with a holomorphic Riemannian metrics $g_{i}$ on each $U_{i}$ such that $g_{i}=f_{ij}g_{j}$ for some holomorphic map $f_{ij}: U_{i}\cap U_{j}\longrightarrow \mathbb{C}$.  Two such covers  
$(\left\{U_{i}\right\}, g_{i})_{i}$ and $(\left\{V_{j}\right\}, h_{j})_{j}$ on $M$ are said to be conformally equivalent if for every $i,j$, there is a holomorphic map $\phi_{ij}: U_{i}\cap V_{j}\longrightarrow \mathbb{C}$ such that $g_{i}=\phi_{ij}h_{j}$ on  $U_{i}\cap V_{j}$. A conformal holomorphic structure on $M$ is then a conformal class of a cover $(\left\{U_{i}\right\}, g_{i})_{i}$. It is said to be conformally flat if it is locally conformally diffeomorphic to $\mathbb{C}^{n}$. Contrary to the real case,  conformal holomorphic Riemannian structures do not derive necessary from holomorphic Riemannian ones. For instance, the complex projective space $\mathbb{CP}^{1}$ admits a conformal holomorphic Riemannian structure but has no  holomorphic Riemannian metric. Another example is provided by the Einstein complex space $\Eins_{n}(\C)$ (see Example \ref{e1} below). Indeed the Fubini-Study metric induces a Kähler metric on $\Eins_{n}(\C)$ (See \cite[Example  10.6]{Kobanomizu}). Thus by \cite{Kobay} (see also \cite{ZeghiSorin}, \cite{Sorindumi}) it does not admit a holomorphic Riemannian metric.

Let $G$ be a Lie group  acting on $M$ by preserving some conformal holomorphic Riemannian structure. The action is said to be essential if $G$ does not preserve any real Riemannian metric on $M$.  This paper aims to classify pairs $\left( M,G\right)$ where $G$ is a complex semi-simple Lie group acting essentially and transitively on $M$. Before going any further, let us start by giving some examples of such pairs.

\subsection{Constructions} 
\label{const}

\subsubsection{The Complex Einstein Universe $\Eins_{n}(\C)$}
\label{e1}
On $\mathbb{C}^{n+2}$, consider the standard holomorphic Riemannian metric $q=dz_{0}^{2}+...+dz_{n+1}^{2}$ and let $\mathsf{Co}_{n+1}(\C)=\lbrace z\in\mathbb{C}^{n+2}: q(z,z)=0 \rbrace$ 
be its light-cone. The complex quadric    $\QQ_n (\C) = (\mathsf{Co}_{n+1} - \{0\})/\C^* \subset \C P^{n+1}$  is the  projectivization of the light-cone \cite[Example  10.6]{Kobanomizu}. The geometry of complex quadrics was amply studied in the litterature in \cite{GasquiGoldschmidt}, \cite{Emilio}, \cite{kups1627}, \cite{Gasquihuber}, \cite{Klein}.\\
The induced metric on $\mathsf{Co}_{n+1}$ is degenerate with kernel the tangent space of  $\C^*$-orbits. It follows that the metric becomes non-degenerate on $\QQ_n(\C)$, but it is defined up to a constant.  Therefore, a holomorphic conformal structure is well defined on $\QQ_n(\C)$. The group $\PSO(n+2, \C)$, which acts transitively on $\QQ_n(\C)$, preserves naturally this  holomorphic conformal structure.  In fact, it is the unique  holomorphic conformal structure  on $\QQ_n(\C)$  preserved by  $\SO(n+2, \C)$. Moreover, the action of $\PSO(n+2, \C)$ is essential. It is called the complex Einstein universe, and denoted $\Eins_n(\C)$. A conformally flat holomorphic conformal structure is then equivalent to giving a $\left( \PSO(n+2, \C), \Eins_n(\C)\right) $-structure.

The stabilizer (of some point) is a parabolic group $P_1$. In fact, $\PSO(n+2, \C)$ acts transitively on  $Gr^0_k$, the space of isotropic $k$-planes. This requires $k \leq $ the integer part of $n/2 + 1$.  Let $P_k$ the stablizer of this action. The parabolic groups $P_k$ are exactly the maximal parabolic 
subgroups of $\PSO(n+2, \C)$ (maximal to mean that only one root space corresponding to a simple root is not contained in such a subgroup). In our investigation in Section \ref{Parabolic_Case1}, we will in particular see that only $Gr^0_1 = \QQ_n(\C)$ admits a $\PSO(n+2, \C)$-invariant holomorphic
conformal structure.

\subsubsection{$\Sp(2n, \C)$-case}
\label{e2}
 The symplectic group $\Sp(2n, \C)$ preserves a (complex) symplectic form 
 $\omega ((x_1, \ldots, x_{2n}), (y_1, \ldots, y_{2n}) ) = \Sigma_{i=1}^{i=n} x_i y_{n+i} - \Sigma_{i=1}^{i=n} y_i x_{n+i}$. So 
 its diagonal action on $\C^{2n} \times \C^{2n}$ preserves the quadratic form on $\C^{4n}$: 
  $$q ((x_1, \ldots, x_{2n}), (y_1, \ldots, y_{2n}) ) = \Sigma_{i=1}^{i=n} x_i y_{n+i} - \Sigma_{i=1}^{i=n} y_i x_{n+i}$$ This determines
  an embedding $\Sp(2n, \C) \to \SO(4n, \C)$.

 \bigskip
  
  Observe   that $\GL(2, \C)$ acts on $\C^{2n} \times \C^{2n}$ by 
$(x, y) \to (a x + b y, c x + d y)$. This action  commutes with the $\Sp(2n, \C)$-action and more generally  with the diagonal action of $\GL(2n, \C)$.
In particular, $\SL(2, \C)$ preserves
the quadratic form $q$, as $q (a x + b y, c x + d y) =  \omega (a x + b y, c x + d y) =
(ad - bc) \omega (x, y) $.

\medskip

Consider now the open simply connected  subset $\mathcal D = \mathcal D_{\Sp(2n, \C)}$ of the quadric  $\QQ_{4n - 2}(\C)$ corresponding to the projectivization of 
the open subset of the $q$-light-cone, 
 $ \{ (x, y)\mid q(x, y) = 0,  \C x \neq \C y \}$.  
The group $\PSp(2n, \C)$ acts transitively and faithfully on it, and we aim to understand its isotropy group, say $Q$.

Let $\mathcal X$ be the space of $\omega$ isotropic 2-planes of $\C^{2n}$. We have a well defined 
$\PSp(2n, \C)$-equivariant map $\pi: \mathcal D  \to \mathcal X$, associating to $(x, y)$ the 2-plane
$\C x \oplus \C y$. The $\pi$-fiber of an $\omega$-isotropic  2-plane $p$ is the set of all its bases
$(b_1, b_2)$, that is $ \C b_1 \oplus \C b_2 = p$. By its true definition, the $\PGL(2, \C)$-action preserves the $\pi$-fibres. If fact $\pi$ is a $\PGL(2, \C)$-principal fibration. In particular, $\PGL(2, \C)$ acts properly and freely on $\mathcal D$.

Let $p = \C e_1 \oplus \C e_{n+1}\in \mathcal X$ where $(e_i)$ is the canonical basis of $\C^{2n}$. Its stabilizer 
 $Q^\prime$ in $\PSp(2n, \C)$ 
preserves the fiber $\mathcal Y= \pi^{-1}(p)$ and acts transitively on it, since the $\PSp(2n, \C)$-action on $\mathcal D$ is transitive and commutes with $\pi$. So on $\mathcal Y$, we have two commuting transitive actions of $Q^\prime$ and $\PGL(2, \C)$.  But, $\mathcal Y$ itself is identified with
$\PGL(2, \C)$, acting on itself on the left (since this action is  free and  transitive). It follows that $Q^\prime$ acts 
on the right on $\mathcal Y$ via a homomorphism $Q^\prime \to \PGL(2, \C)$. Since $\PGL(2, \C)$ is semi-simple, this homomorphism splits, up 
to finite index, and thus, up to finite index $Q^\prime = \PGL(2, \C) \ltimes Q $, where $Q$ is the kernel of  $Q^\prime \to \PGL(2, \C)$.

Clearly  $Q$ acts trivially on $\mathcal Y$.  
In fact $Q$ is the stabilizer 
for  the $\PSp(2n, \C)$-action on $\mathcal D$ of any point of the fiber $\mathcal Y$.  Therefore, $\mathcal D$ as a homogeneous space
can be identified to $\PSp(2n, \C)/Q$.

Since $\mathcal X$ is compact, $Q^\prime$ is a parabolic subgroup of $\PSp(2n, \C)$, and in particular the normalizer of 
$Q$ is parabolic. To finish,  take  $H$  to be  a semi-direct product 
$\Gamma \ltimes Q$, where $\Gamma$ is a co-compact lattice in $\PGL(2, \C)$. Then $H \subset Q^\prime$ with identity component $H^0 = Q$, $M_{1}=\PSp(2n, \C)/H$ is compact and covered by $\mathcal D = \PSp(2n, \C)/ Q$. 
  
\subsubsection{$\SL(n, \C)$-case}
\label{e3}
Given an $n$-dimensional complex vector space $E$. The diagonal action of $\GL(E)$  on $E \times E^*$
preserves the quadratic form $q(x, f) = f(x)$. In addition, the $\PSL(E)$-action is transitive and faithful 
on $\QQ(E \times E^*)$, the projectivization of $\{ (x, f) \mid  f(x) = 0, (x,f) \neq (0,0)\}$. 
 
 Let $Q$ be the stabilizer of a point in the open simply connected  subset  $\mathcal D_{\SL(n, \C)}$ of the quadric $\QQ(E \times E^*)$ corresponding to the projectivization of 
the open subset of the $q$-light-cone, 
 $ \{ (x, f)\mid f(x) = 0,  x\neq 0, f\neq 0 \}$. 
 It has codimension 1 in its normalizer $P$.  To see this, let $e_1, \ldots, e_n$ be a basis of $E$ and $e_1^*, \ldots, e_n^*$
 its dual bases. Consider $p$ the point in the projective space corresponding to $(e_1, e_n^*) \in \mathcal D$.  Its stablizer $Q$ consists 
 of matrices of the form
 $ \begin{pmatrix}  \lambda &  u^{t} & v \\
 0 &  D & C \\
 0 & 0 &  \frac{1}{\lambda}
 \end{pmatrix}
 $, where $u$  is a vector of dimension $n-2$,  $\lambda, v$ are scalars, $D$ is a $(n-2) \times (n-2)$-matrix, and $C$ is a vector of dimension $n-2$, such that $  \det D = 1$. Its normaliser $Q^\prime$
 consists of matrices of the form  $ \begin{pmatrix}  \lambda &  u^{t} & v \\
 0 &  D & C \\
 0 & 0 &  \lambda^\prime
 \end{pmatrix}
 $, with $\lambda (\det D) \lambda^\prime = 1$.  This is the stablizer of the flag $(\C e_1, \C e_1 \oplus \ldots  \oplus \C e_{n-1})$ and hence is parabolic.
 The quotient group $Q^\prime / Q$ has  dimension 1. More precisely,  
 up to a finite index, $Q^\prime $ is a semi-direct product $L \ltimes Q$, where $L \cong \C^*$ is represented as
 matrices of the form $ \begin{pmatrix}  \alpha &  0 & 0 \\
 0 &  \alpha^{-2} &  0\\
 0 & 0 &  \alpha 
 \end{pmatrix}
 $. If $\Gamma$ is a lattice in $\C^{*}$, then, $H = \Gamma \ltimes Q$ yields a compact quotient $M_{2}=\PSL(n, \C) / H$ covered by $\mathcal D =\PSL(n, \C) / Q$.

 \begin{Remark}[Uniqueness] Although we will not need it, let us observe that 
 in both cases, the invariant domains $\mathcal D$ are unique. More precisely, there are unique (irreducible) representations 
$ \Sp(2n, \C) \to \SO(4n, \C)$,  and $\SL(n, \C) \to \SO(2n, \C)$.  Both have a unique dense invariant domain $\mathcal D_{\Sp(2n)}$ (resp. $\mathcal D_{\SL_n}$). 
\end{Remark}
\subsection{Rigidity, main result}
D'Ambra and Gromov conjectured in \cite{Gromov} that compact pseudo-Riemmannian conformal manifolds with an essential action of the conformal group are conformally flat. This conjecture, often known as the pseudo-Riemannian Lichnerowicz conjecture, was  later disproved by Frances in \cite{Francesun}. Additionally, this conjecture has been studied under signature restiction, in the works of Zimmer, Bader, Nevo,  Frances, Zeghib,  Melnick and Pecastaing (see \cite{Zimmerun}, \cite{Bader}, \cite{Francesquatre}, \cite{Pecastaingun}, \cite{Pecastaingdeux}, \cite{Vincenttrois}, \cite{Vincentquatre}). The present paper is the second in a series, exploring the Lichnerowicz conjecture in the  homogeneous context. In \cite{BDRZ1} we provided a positive affirmation of the conjecture when  the non compact semi-simple component of the conformal group is  the Möbius group. This article deals with the homogeneous Lichnerowicz conjecture in the complex (or real split) cases. More precisely,  we will show that the  examples constructed in Section \ref{const} are essentially the only ones:

\begin{theo}
\label{theo1}
Let $M$ be a compact connected complex manifold  endowed with a faithful conformal holomorphic Riemannian  structure   invariant under an essential and transitive  action of a complex semi-simple Lie group $G$.  Then $M$ is conformally flat. Furthermore:

-  If $M$ is simply connected, then, we have one of the following situations: 
\begin{enumerate}
\item $G=\PSO(n+2,\C)$ and $M=\Ein_{n}(\C)$ with $n\geq 1$   (in particular for $n = 1$,  $G=\PSL(2,\C)$ and $M=\mathbb{CP}^{1}$,  and for $n = 2$, 
$G=\PSL(2,\C)\times \PSL(2,\C)$ and $M= \mathbb{CP}^{1}\times \mathbb{CP}^{1}$)
or;
\item $G $ is the exceptional group $G_{2}$ and $M=\Ein_{5}(\C)$

\end{enumerate}

- If $M$ is not simply connected, then it   fits into one of the examples above in Section \ref{const}.  In particular: 
\begin{enumerate}
 
\item $G=\PSp(2n,\C)$ and $M$ is a quotient of a $\PSp(2n, \C)$-homogeneous open subset in $\Eins_{2n-2}(\C)$ ($n \geq 3$).
The fundamental group $\pi_1(M)$ is a co-compact lattice in $\PGL(2, \C)$ (i.e.   the fundamental group of a closed hyperbolic 3-manifold).

\item $G=\PSL(n,\C)$ and $M$  a quotient of a $\PSL(n, \C)$-homogeneous open subset in $\Eins_{2n-2}(\C)$ ($n \geq 3$).  
The fundamental group $\pi_1(M)$ is infinite cyclic.

\end{enumerate}
\end{theo}



\subsection{Organization of the  article}
The paper is organized as follows: In Section \ref{sect1}, we provide an algebraic formulation of our initial problem using Lie algebra terminology. Section \ref{sect2} delves into a detailed examination of the structure of the isotropy sub-algebra. We will specifically distinguish between three different cases based on the size of the isotropy sub-algebra. Sections \ref{sect3}, \ref{sect4}, and \ref{Parabolic_Case} are dedicated to proving the classification theorem in these distinct cases.
\section{Algebraic formulation}
\label{sect1}
Assume that $M$ is endowed with a conformal holomorphic Riemannian structure $\mathcal{G}$ invariant under the action of a complex semi-simple Lie group $G$. We will assume in addition that $G$ acts  transitively and essentially on $(M,\mathcal{G})$. 

Let $x_{0}$ be a fixed point of $M$ and denote by $H$ its stabilizer in $G$ so that $T_{x_{0}}M$ is identified with $\mathfrak{g}/\mathfrak{h}$. The conformal structure $\mathcal{G}$ defines a conformal class of a non-degenerate complex bilinear symmetric form  $g$ on $\mathfrak{g}/\mathfrak{h}$ which in turn gives rise to a conformal class of a degenerate complex bilinear symmetric form $\left\langle .,.\right\rangle$ on $\mathfrak{g}$ admitting $\mathfrak{h}$ as a kernel. More precisely, the  form $\left\langle .,.\right\rangle$ is defined by $$\left\langle X,Y\right\rangle=g\left(X^{*}(x),Y^{*}(x)\right),$$ where $X^{*}$, $Y^{*}$ are the fundamental vector fields associated to $X$ and $Y$.

Consider $P=\overline{H}^{Zariski}$ the Zariski closure of the isotropy group $H$. It preserves the conformal class of $\left\langle .,.\right\rangle$. More precisely, there is a morphism $\delta:P\longrightarrow \mathbb{C}^{*}$ such that for every $p\in P$ and every $u,v\in \mathfrak{g}$, 
\begin{equation}
\label{equationdeux}
\left\langle\operatorname{Ad}_{p}(u), \operatorname{Ad}_{p}(v)\right\rangle=\delta(p)\left\langle u,v\right\rangle=\left(\operatorname{det}\left(\operatorname{Ad}_{p}\right)_{\vert{\mathfrak{g}/\mathfrak{h}}}\right)^{\frac{2}{n}}\left\langle u,v\right\rangle.
\end{equation}
In particular the group $P$ normalizes $H$.

Differentiating Equation \ref{equationdeux}, we get a linear function, that we continue to denote $\delta$, from $\mathfrak{p}$ the Lie algebra of $P$ to  $\mathbb{C}$ such that for every $p\in \mathfrak{p}$ and every $u,v\in \mathfrak{g}$
\begin{equation}
\label{equationcinq}
\left\langle\operatorname{ad}_{p}(u),v\right\rangle + \left\langle u, \operatorname{ad}_{p}(v)\right\rangle=\delta(a)\left\langle u,v\right\rangle.
\end{equation}

In particular if $p\in  \mathfrak{p}$ preserves the metric then $\delta(p)=0$ and 

\begin{equation}
\label{equationsix}
\left\langle\operatorname{ad}_{p}(u),v\right\rangle + \left\langle u, \operatorname{ad}_{p}(v)\right\rangle=0.
\end{equation}

As $\mathfrak{p}$ is a complex uniform algebraic sub-algebra of the semi-simple algebra $\mathfrak{g}$, there exists a Cartan sub-algebra $\mathfrak{a}$ of $\mathfrak{g}$ together with an ordered root system $\Delta=\Delta^{-}\sqcup \Delta^{+}$ and a root space decomposition $\mathfrak{g}=   \bigoplus_{\alpha\in \Delta^{-}} \mathfrak{g}_{\alpha}   \oplus \mathfrak{g}_{0}\oplus \bigoplus_{\alpha\in \Delta^{+}} \mathfrak{g}_{\alpha}=\mathfrak{g}_{-}\oplus\mathfrak{a}\oplus \mathfrak{g}_{+}$ such that $\mathfrak{a}\oplus \mathfrak{g}_{+} \subset\mathfrak{p}$ \cite[Corrolaire~16.13]{borel}.

\begin{defi}
Two elements $\alpha$, $\beta$ of $\Delta\cup 0$ are said to be paired if $\mathfrak{g}_{\alpha}$ and $\mathfrak{g}_{\beta}$ are not $\left\langle .,.\right\rangle-$orthogonal. 
\end{defi}

Note that because $\left\langle .,.\right\rangle$ is non trivial, there  always exist two paired elements (possibly the same) $\alpha$, $\beta$ of  $\Delta\cup 0$. Any such elements $\alpha$ and $\beta$ verify $\alpha+\beta=\delta$. This shows  that for any element $\alpha$  there is at most one $\beta$ (depending whether $\mathfrak{g}_{\alpha}\subset\mathfrak{h}$ or not) paired with it. Moreover:

\begin{prop}
We have:
\begin{enumerate}
\item $\mathfrak{h}$ is a non trivial ideal of $\mathfrak{p}$;
\item $\mathfrak{p}\subsetneq \mathfrak{g}$;
\item The restriction of $\delta$ to $\mathfrak{a}$ is a non trivial  linear form.
\end{enumerate}
\end{prop}

\begin{defi}
The restriction of $\delta$ to $\mathfrak{a}$ is called  \textbf{distortion}.
\end{defi}

\begin{proof}
$1)$ By  \cite[Proposition~2.6]{BDRZ1},  $\mathfrak{h}$ is non trivial.\\
$2)$ Suppose the converse. Since $\mathfrak{h}$ is a non trivial ideal of $\mathfrak{p}$, Equation \ref{equationsix} is verified for every $u,v\in \mathfrak{g}$ and every $p\in \mathfrak{h}$ which contradicts the essentiality of the action. \\
$3)$  Now as $\mathfrak{g}_{-}$ and $\mathfrak{g}_{+}$ are nilpotent sub-algebras, we have that $\delta$ is trivial on $(\mathfrak{p}\cap\mathfrak{g}_{-})\oplus \mathfrak{g}_{+}$. If $\delta$ was trivial on $\mathfrak{a}$ then $\delta$ would be trivial on $\mathfrak{p}=(\mathfrak{p}\cap\mathfrak{g}_{-})\oplus \mathfrak{a} \oplus \mathfrak{g}_{+}$ which clearly contradicts the essentiality hypothesis.\\
\end{proof}

In the rest of this paper we will abandon our original group formulation and instead adopt the following Lie algebra one: 

\begin{itemize}
\item There is a  root space decomposition as above,
\item There is a distortion $\delta: \mathfrak{a} \longrightarrow \mathbb{C}$, 
\item The pairing condition of two weight spaces implies their sum is $\delta$, 
\item The essentiality is translated into the fact that 
$\delta \neq 0$, and the compactness of $G/H$ is replaced by the fact 
that  $\mathfrak{a} \oplus\mathfrak{g}_{+}$ normalizes $\h$.
\end{itemize}
We finish this section by the following useful definition:
\begin{defi}
We say that a subalgebra $\mathfrak{g}'$ is a modification of $\mathfrak{g}$, if $\mathfrak{g}'$ projects surjectively on $\mathfrak{g}/\mathfrak{h}$. Equivalently, $M=G'/(G'\cap H)$, where  $G'$ is the connected subgroup of $G$ associated to $\mathfrak{g}'$.
\end{defi}
\section{Structure of the isotropy sub-algebra: Synthetic study}
\label{sect2}
In this part we will study in detail the structure of the sub-algebra $\mathfrak{h}$. Let us start with the following proposition:

\begin{prop}
\label{Proposition1}
We have:
\begin{enumerate}
\item If $\mathfrak{a}\subset \mathfrak{h}$, then the Borel sub-algebra $\mathfrak{b}=\mathfrak{a}\oplus\mathfrak{g}_{+}$ is contained in $\mathfrak{h}$;
\item If $\mathfrak{a}\nsubseteq \mathfrak{h}$, then $\delta$ is a root paired with $0$. In particular, $\mathfrak{g}_{\delta}$ is not contained in $\mathfrak{p}$. Moreover, the sub-algebra $\mathfrak{a}\cap\mathfrak{h}$ has co-dimension one in $\mathfrak{a}$. 
\end{enumerate}
\end{prop}

\begin{proof}
Suppose first that $\mathfrak{a}\subset \mathfrak{h}$. Then $\mathfrak{g}_{+}=\left[\mathfrak{a},\mathfrak{g}_{+}\right]\subset \left[\mathfrak{h},\mathfrak{p}\right]\subset \mathfrak{h}$. This implies that the Borel sub-algebra $\mathfrak{b}=\mathfrak{a}\oplus\mathfrak{g}_{+}\subset \mathfrak{h}$.

If on the contrary, $\mathfrak{a}$ is not contained in $\mathfrak{h}$ then $0$ is paired with $\delta$ and hence $\delta$ is a root. Let $p\in \mathfrak{g}_{\delta}\cap \mathfrak{p}$ and $u=v$ in $\mathfrak{a}$. Substituting this into Equation \ref{equationcinq}, we obtain $\delta(u)\left\langle p,u\right\rangle=0$ for every $p\in \mathfrak{g}_{\delta}$ and $u\in \mathfrak{a}$. Thus by density we get $\left\langle\mathfrak{a}, \mathfrak{g}_{\delta}\right\rangle=0$ which contradicts the fact that $\delta$ is paired with $0$. So $\mathfrak{g}_{\delta}\cap \mathfrak{p}=\emptyset$

As $\mathfrak{g}_{\delta}$ is of dimension one and $\mathfrak{h}$ is the kernel of $\left\langle .,.\right\rangle$  we get that $\mathfrak{a}\cap\mathfrak{h}$ is of codimension one in $\mathfrak{a}$.
\end{proof}

\subsection{Case one: $\mathfrak{a}\nsubseteq \mathfrak{h}$} Then:
\begin{prop}
\label{Prop1}
Up to modification $\mathfrak{g}$ is simple.
\end{prop}
\begin{proof}
Assume that $\mathfrak{g}=\mathfrak{g}_{1}\oplus\mathfrak{g}_{2}$ is the direct sum of a simple Lie algebra $\mathfrak{g}_{1}\nsubseteq \mathfrak{h}$ and a semi-simple Lie algebra $\mathfrak{g}_{2}\nsubseteq \mathfrak{h}$. Thus there exist a root $\alpha$ of $\mathfrak{g}_{1}$ and a root $\beta$  of $\mathfrak{g}_{2}$ such that $\mathfrak{g}_{\alpha}\nsubseteq \mathfrak{h}$ and  $\mathfrak{g}_{\beta}\nsubseteq \mathfrak{h}$. Therefore, $\delta-\alpha$, $\delta-\beta$ are also roots of $\mathfrak{g}$. But the roots of $\mathfrak{g}$ are the disjoint union of the those of  $\mathfrak{g}_{1}$ and $\mathfrak{g}_{2}$. This implies that $\delta$ is a root of both $\mathfrak{g}_{1}$ and $\mathfrak{g}_{2}$ which is a contradiction.
\end{proof}

By \cite[Proposition~2.17]{K}, for every root $\alpha$ there exists an element $H_{\alpha}\in \mathfrak{a}$ such that $\operatorname{B}(H_{\alpha},.)=\alpha$, where here $\operatorname{B}$ is the non degenerate Killing form of $\mathfrak{a}$.

Let  $p\in \mathfrak{g}_{-\delta}$ and choose  $0\neq u\in \mathfrak{g}_{\delta}$ such that $\left[p,u\right]=H_{\delta}$. Applying Equation \ref{equationsix} with $p,u=v$  we obtain $\left\langle H_{\delta}, u\right\rangle=0$ and hence  $\left\langle H_{\delta}, \mathfrak{g}_{\delta}\right\rangle=0$. However, by Proposition \ref{Proposition1}, $\delta$ is a root paired with $0$. Therefore $H_{\delta}\in \mathfrak{a}\cap\mathfrak{h}$. Now we have the following important Lemma:

\begin{lemm}
\label{Lemm2}
Let $\alpha$ be a root, which we will assume to be positive. Then
\begin{enumerate}
\item If $\delta(H_{\alpha})\neq 0$, $\mathfrak{g}_{\alpha}$ is contained in $\mathfrak{h}$;
\item If $\delta(H_{\alpha})= 0$ and $\delta-\alpha$ is a root, $\mathfrak{a}\cap\mathfrak{h}=H_{\alpha}^{\perp}$, where the orthogonality is  with respect to the Killing form $\operatorname{B}$. In particular such $\alpha$ is unique.
\item If $\delta(H_{\alpha})= 0$, $\mathfrak{g}_{-\alpha}\oplus\mathbb{C}H_{\alpha}\oplus \mathfrak{g}_{\alpha}$ preserves $\left\langle .,.\right\rangle$. In particular if $\mathfrak{g}_{\alpha}\subset \mathfrak{h}$ then $\mathfrak{g}_{-\alpha}\oplus\mathbb{C}H_{\alpha}\oplus \mathfrak{g}_{\alpha}\subset \mathfrak{h}$.
\end{enumerate}
\end{lemm}

\begin{proof}
First assume that $\delta(H_{\alpha})\neq 0$. Thus $\mathfrak{g}_{\alpha}=\delta(H_{\alpha})\mathfrak{g}_{\alpha}=\left[H_{\delta},\mathfrak{g}_{\alpha}\right]$. But $\mathfrak{h}$ is an ideal of $\mathfrak{p}$, $H_{\delta}\in \mathfrak{h}$ and $\mathfrak{g}_{\alpha}\subset \mathfrak{p}$. Therefore $\mathfrak{g}_{\alpha}\subset\mathfrak{h}$.

Assume on the contrary that $\delta(H_{\alpha})= 0$ and $\delta-\alpha$ is a root. Take $H\in H_{\alpha}^{\perp}$ so that $\alpha(H)=0$. On the one hand, using Equation \ref{equationsix}, with $p\in \mathfrak{g}_{\alpha}$, $u=H$ and $v\in \mathfrak{g}_{\delta-\alpha}$ gives us $\left\langle H,\left[p,v\right]\right\rangle=0$. However, according to \cite[Corollary~2.35]{K}, $\left[\mathfrak{g}_{\alpha}, \mathfrak{g}_{\delta-\alpha}\right]=\mathfrak{g}_{\delta}$, implying $H\in \mathfrak{a}\cap\mathfrak{h}$. On the other hand, Proposition \ref{Proposition1} tells us that $H_{\alpha}^{\perp}$ and $ \mathfrak{a}\cap\mathfrak{h}$ have the same dimension. Thus   $H_{\alpha}^{\perp}=\mathfrak{a}\cap\mathfrak{h}$.

To finish, assume just that $\delta(H_{\alpha})= 0$. Then $\mathbb{C}H_{\alpha}\oplus \mathfrak{g}_{\alpha}$ preserves $\left\langle .,.\right\rangle$. Thus its orbit  under the action of $\mathfrak{g}_{-\alpha}\oplus\mathbb{C}H_{\alpha}\oplus \mathfrak{g}_{\alpha}\cong \mathfrak{sl}(2,\mathbb{C})$ is compact and hence trivial by \cite[Lemma~2.7]{BDRZ1}. 

If $\mathfrak{g}_{\alpha}\subset \mathfrak{h}$, then since $\mathfrak{h}$ is an ideal of the sub-algebra preserving the conformal class of $\left\langle .,.\right\rangle$, we have  $\mathfrak{g}_{-\alpha}\oplus\mathbb{C}H_{\alpha}\oplus \mathfrak{g}_{\alpha}=\left[\mathfrak{g}_{\alpha},\mathfrak{g}_{-\alpha}\oplus\mathbb{C}H_{\alpha}\oplus \mathfrak{g}_{\alpha}\right]\subset \mathfrak{h}$.
\end{proof}

For every root $\alpha$, let us fix two elements $u_{\alpha}\in\mathfrak{g}_{\alpha} $ and $u_{-\alpha}\in \mathfrak{g}_{-\alpha}$ such that $\left[u_{\alpha},u_{-\alpha}\right]=H_{\alpha}$. Let $\alpha$, $\beta$ two roots such that $\alpha+\beta$ is a root. By \cite[Corollary~2.35]{K} we have that $\left[\mathfrak{g}_{\alpha},\mathfrak{g}_{\beta}\right]=\mathfrak{g}_{\alpha+\beta}$. Therefore, there is a non zero complex number $k_{\alpha,\beta}$ such that  $\left[u_{\alpha},u_{\beta}\right]=k_{\alpha,\beta}u_{\alpha+\beta}$. Now, if $\alpha$ is a  root such that $\mathfrak{g}_{\alpha}\nsubseteq \mathfrak{h}$, then $\delta-\alpha$ is also a root. By assuming $\alpha$ negative if necessary, we use Equation \ref{equationsix}, with $p=u_{-\alpha}$, $u=u_{\alpha}$ and $v=u_{\delta}$ to obtain: $\left\langle u_{\alpha}, u_{\delta-\alpha}\right\rangle=\frac{1}{k_{\alpha,\delta}}\left\langle H_{\alpha}, u_{\delta} \right\rangle$. As a consequence we get the following uniqueness property:

\begin{prop}
\label{proposition5}
The conformal class of $\left\langle .,.\right\rangle$ depends only on $\mathfrak{a}\cap\mathfrak{h}$ and $\mathfrak{g}_{\delta}$.
\end{prop}
\subsection{Case two: $\mathfrak{a}\subset \mathfrak{h}$}
In this case $\delta$ is no longer a root, rather a sum of two roots. We have:
\begin{prop}
\label{Prop2}
Up to modification, $\mathfrak{g}$ is: 
\begin{itemize}
\item Simple or;
\item The direct sum of two rank one complex simple Lie algebras.
\end{itemize}

\end{prop}
\begin{proof}
Assume we are not in the second case. So one can write  $\mathfrak{g}=\mathfrak{g}_{1}\oplus\mathfrak{g}_{2}$ as the direct sum of simple Lie algebra $\mathfrak{g}_{1}\nsubseteq \mathfrak{h}$ and a semi-simple one $\mathfrak{g}_{2}$. Moreover, if  $\mathfrak{g}_{2}\nsubseteq \mathfrak{h}$ then there are a root $\alpha$ of $\mathfrak{g}_{1}$ and two roots $\beta\neq\gamma$  of $\mathfrak{g}_{2}$ such that $\mathfrak{g}_{\alpha}\nsubseteq \mathfrak{h}$, $\mathfrak{g}_{\beta}\nsubseteq \mathfrak{h}$ and $\mathfrak{g}_{\gamma}\nsubseteq \mathfrak{h}$. Consequently, $\delta-\alpha$, $\delta-\beta$ and $\delta-\gamma$ are also roots of $\mathfrak{g}$. However, this is impossible since the roots of $\mathfrak{g}$ are the union of the roots of  $\mathfrak{g}_{1}$ and $\mathfrak{g}_{2}$. Thus $\mathfrak{g}_{2}$ must be in $\mathfrak{h}$ and hence $\mathfrak{g}$ is, up to modification, simple.
\end{proof}

\section{The $\operatorname{Sp}(n,\mathbb{C})$ case}
\label{sect3}

In this part we will prove Theorem \ref{theo1} when $\mathfrak{a}\nsubseteq \mathfrak{h}$ and $\mathfrak{g}_{+}\nsubseteq \mathfrak{h}$. By Proposition \ref{Prop1}, up to modification, the Lie algebra $\mathfrak{g}$ is simple. The root systems associated to a simple complex Lie algebra are well known and classified. They are of $A_{n}$, $B_{n}$, $C_{n}$ and $D_{n}$ types as well as the exceptional ones $E_{6}$, $E_{7}$, $E_{8}$, $F_{4}$ and $G_{2}$. Up to isomorphism, they are described by means of the canonical basis of $\mathbb{R}^{n}$. Detailed descriptions of these root systems, along with their associated canonical simple roots, can be found in \cite{K}. From now and till the end of the paper we will assume, up to isomorphism, that the root system $\Delta$ is a canonical root system endowed with its canonical order. 
The notations and terminology used here follow \cite[Appendix C]{K}.

In this case $\delta$ is a root and  there exists a positive root $\alpha$ such that $\mathfrak{g}_{\alpha}\nsubseteq \mathfrak{h}$. Hence $\delta-\alpha$ is also a root. By Lemma \ref{Lemm2}, we have that $\delta(H_{\alpha})= 0$, $\mathfrak{a}\cap\mathfrak{h}=H_{\alpha}^{\perp}$ and $\alpha$ is unique. We have:
\begin{prop}
\label{prop20}
Let $\beta$ be a positive root different from $\alpha$. If $\beta$ is not orthogonal to $\alpha$ then $\mathfrak{g}_{-\beta}\nsubseteq \mathfrak{h}$.
\end{prop}
\begin{proof}
Assume that $\mathfrak{g}_{-\beta}\subset \mathfrak{h}$. As $\beta\neq \alpha$ then by Lemma \ref{Lemm2},  $\mathfrak{g}_{\beta}\subset \mathfrak{h}$ and hence $\mathbb{C}H_{\beta}=\left[ \mathfrak{g}_{-\beta}, \mathfrak{g}_{\beta}\right]\subset \mathfrak{a}\cap\mathfrak{h}=H_{\alpha}^{\perp}$. This means that $\beta$ is orthogonal to $\alpha$.
\end{proof}

Now we have the following proposition:
\begin{prop}
Up to the action of the Weyl group, the pairs of roots $(-\delta,\alpha)$ such that: $\alpha$ is orthogonal to $\delta$ and $\delta-\alpha$ is a root are:
\begin{itemize}
\item $B_{n}$: $(-\delta,\alpha)=(e_{1},e_{2})$;
\item $C_{n}$: $(-\delta,\alpha)=(e_{1}+e_{2},e_{1}-e_{2})$;
\item $F_{4}$: $(-\delta,\alpha)=(e_{1},e_{2})$ or $(-\delta,\alpha)=\left(  \frac{1}{2}(e_{1}+e_{2}-e_{3}-e_{4}),\frac{1}{2}(e_{1}+e_{2}+e_{3}+e_{4})\right) $.
\end{itemize}

\end{prop}

\begin{proof}
As $\delta$ is orthogonal to $\alpha$ we have: 
\begin{equation}
\label{equation9}
\left|\delta-\alpha\right|^{2}=\left|\delta\right|^{2}+\left|\alpha\right|^{2}.
\end{equation}

First, assume that our root system is of type $A_{n}$, $D_{n}$, $E_{6}$, $E_{7}$, or $E_{8}$ . In this cases all the roots have the same length. Putting this in Equation \ref{equation9} gives us a contradiction.

Now, if we are in the $G_{2}$ type. Then we have $12$ roots: six of them have length $2$ and the other six  have length $6$. Again, these do not verify Equation \ref{equation9}. 

Finally in all the remaining types ($B_{n}$, $C_{n}$ and $F_{4}$) we can verify easily that such pairs exist. We then use   the action of the Weyl group to conclude.
\end{proof}

We are left with three types of root systems. Namely $B_{n}$, $C_{n}$ and $F_{4}$.
\begin{prop}
The pair $(-\delta,\alpha)$ exists only in the root systems of type $C_{n}$.
\end{prop}
\begin{proof}
We first prove that the $B_{n}$ case is impossible. Assume $n>2$, $(-\delta,\alpha)=(e_{i},e_{j})$ and let $\beta=e_{j}+e_{k}$ with $i\neq j\neq k$. As $\delta-\beta=-e_{i}-e_{j}-e_{k}$ is not a root, we have that $\mathfrak{g}_{\beta}\subset \mathfrak{h}$. On the other hand $-e_{k}$ is orthogonal to $\delta$ and $\mathfrak{g}_{e_{k}}\subset \mathfrak{h}$ so by Lemma \ref{Lemm2}, $\mathfrak{g}_{-e_{k}}\subset \mathfrak{h}$. Thus $\left[\mathfrak{g}_{\beta},\mathfrak{g}_{-e_{k}}\right]=\mathfrak{g}_{\alpha}\subset \mathfrak{h}$ which is a contradiction. So $n$ must be equal to $2$ and $B_{2}=C_{2}$ \cite[Pages~26-27]{Serre}.

As for the $F_{4}$ case, the same proof works.

\end{proof}

The only remaining case is the $C_{n}$ type. In this case we have only one possibility for the pairs $(-\delta,\alpha)$. Namely:
\begin{prop}
$(-\delta,\alpha)=(e_{1}+e_{2},e_{1}-e_{2})$.
\end{prop}
\begin{proof}
By contradiction, assume that $(-\delta,\alpha)=(e_{i}+e_{j},e_{i}-e_{j})$ for some $1\leq i<j\leq n$ such that $i\neq 1$ or $j\neq 2$. If $i\neq 1$, then $\beta=e_{1}-e_{j}\neq \alpha$ is a positive root  which is not orthogonal to $\alpha$. Thus by Proposition \ref{prop20}, $\mathfrak{g}_{-\beta}\nsubseteq \mathfrak{h}$ and hence $\delta+\beta=-e_{i}-2e_{j}+e_{1}$ is also a negative root which is clearly false. If in contrast $j\neq 2$ then take $\beta=e_{2}-e_{j}$ and the same proof works.
\end{proof}

The fact that we already have an example of such type (Example \ref{e2}) together with the uniqueness property in Proposition \ref{proposition5} give us:
\begin{cor}
If $\mathfrak{a}\nsubseteq \mathfrak{h}$ and $\mathfrak{g}_{+}\nsubseteq \mathfrak{h}$ then $G=\Sp(n,\C)$ and $M=M_{1}$. In particular $M$ is conformally flat.
\end{cor}

\section{The $\operatorname{SL}(n,\mathbb{C})$ case}
\label{sect4}
In this part, we will prove Theorem \ref{theo1} when $\mathfrak{a}\nsubseteq \mathfrak{h}$ and $\mathfrak{g}_{+}\subset \mathfrak{h}$. In this case $\delta$ is a negative root. Let $\alpha$ be a positive root such that $\delta-\alpha$ is also a root. Consequently,  $\mathfrak{g}_{\delta-\alpha}\subset \mathfrak{h}$. If this were not the case, then $\delta-\alpha$ would be paired with $\alpha$, leading to a contradiction. Now, on the one hand $\mathfrak{g}_{\delta}=\left[\mathfrak{g}_{\delta-\alpha},\mathfrak{g}_{\alpha} \right] \subset \mathfrak{h}\subset \mathfrak{p}$. On the other hand, according to Proposition \ref{Proposition1}, $\mathfrak{g}_{\delta}\cap \mathfrak{p}=\lbrace 0\rbrace$. This leads to a contradiction. Thus:
\begin{prop}
The negative root $\delta$ is the minimal root.
\end{prop}
As a consequence we get:
\begin{prop}
The only possible type is $A_{n}$. In particular $-\delta=e_{1}-e_{n+1}$.
\end{prop}
\begin{proof}
First assume that we are in the $B_{n}$ type. In this case, we have $\delta=-e_{1}-e_{2}$. Here $e_{1}-e_{2}$, $e_{i}$ for $i\geq 3$ are all orthogonal to $\delta$. Using Lemma \ref{Lemm2}, this implies that $H_{e_{1}-e_{2}}$ and $H_{e_{i}}$ for $i\geq 3$ belong to $\mathfrak{a}\cap \mathfrak{h}$. As $H_{\delta}\in \mathfrak{a}\cap \mathfrak{h}$, we get that $\mathfrak{a}\cap \mathfrak{h}=\mathfrak{a}$ which contradicts Proposition \ref{Proposition1}. 

The same proof works for the $C_{n}$ and $D_{n}$ types.

In the exceptional case $E_{6}$,  $-\delta=\frac{1}{2}(e_{8}-e_{7}-e_{6}+e_{5}+e_{4}+e_{3}+e_{2}+e_{1})$. On the one hand $\delta+\alpha_{1}$, $\delta+\alpha_{3}$, $\delta+\alpha_{4}$, $\delta+\alpha_{5}$ and $\delta+\alpha_{6}$ are not roots. So $\mathfrak{g}_{-\alpha_{1}}$, $\mathfrak{g}_{-\alpha_{3}}$, $\mathfrak{g}_{-\alpha_{4}}$, $\mathfrak{g}_{-\alpha_{5}}$, $\mathfrak{g}_{-\alpha_{6}}$ are all in $\mathfrak{h}$. This shows that $H_{\alpha_{1}}$, $H_{\alpha_{3}}$, $H_{\alpha_{4}}$, $H_{\alpha_{5}}$, $H_{\alpha_{6}}$ are all in $\mathfrak{a}\cap \mathfrak{h}$. On the other hand, $H_{\delta}\in \mathfrak{a}\cap \mathfrak{h}$. But $\delta$, $\alpha_{1}$, $\alpha_{3}$, $\alpha_{4}$, $\alpha_{5}$, $\alpha_{6}$ are linearly independent. Thus $\mathfrak{a}\cap \mathfrak{h}=\mathfrak{a}$ which contradicts Proposition \ref{Proposition1}. 

In the exceptional case $E_{7}$, $-\delta=e_{8}-e_{7}$. In this case for every $1\leq i\leq 7$,  $\delta+\alpha_{i}$ is not a root. This means that all the $\mathfrak{g}_{-\alpha_{i}}$ are  in $\mathfrak{h}$. Hence $\mathfrak{g}=\mathfrak{h}$ which is a contradiction.

In the exceptional case $E_{8}$, $-\delta=\frac{1}{2}(e_{8}+e_{7}+e_{6}+e_{5}+e_{4}+e_{3}+e_{2}+e_{1})$ and the same proof as in exceptional case $E_{6}$ works here too. 

Now let us consider the exceptional case $G_{2}$. Here $-\delta=2e_{3}-e_{2}-e_{1}$. Consequently, $\delta+\alpha_{1}$ is not a root and hence  $\mathfrak{g}_{-\alpha_{1}}\subset \mathfrak{h}$. Thus $H_{\alpha_{1}}\in \mathfrak{a}\cap \mathfrak{h}$. Together with the fact that $H_{\delta}\in \mathfrak{a}\cap \mathfrak{h}$, we conclude that $\mathfrak{a}\cap \mathfrak{h}=\mathfrak{a}$ which is in contradiction with Proposition \ref{Proposition1}.

To conclude, let's consider the exceptional case $F_{4}$. Here we also have $-\delta=e_{1}+e_{2}$. Consequently,   $\delta+\alpha_{1}$, $\delta+\alpha_{2}$, and $\delta+\alpha_{3}$ are not roots. This implies that $\mathfrak{g}_{-\alpha_{1}}$, $\mathfrak{g}_{-\alpha_{2}}$, $\mathfrak{g}_{-\alpha_{3}}$ are all in $\mathfrak{h}$ and therefore $H_{\alpha_{1}}$, $H_{\alpha_{2}}$ and $H_{\alpha_{2}}$ are  in $\mathfrak{a}\cap \mathfrak{h}$.  Together with the fact that $H_{\delta}\in \mathfrak{a}\cap \mathfrak{h}$, we deduce that $\mathfrak{a}\cap \mathfrak{h}=\mathfrak{a}$, which once more  contradicts Proposition \ref{Proposition1}.
\end{proof}
In the remaining $A_{n}$ case, the sub-algebra $\mathfrak{a}\cap\mathfrak{h}$ is completely determined by the root $\delta$. Indeed, $-\delta=e_{1}-e_{n+1}$ and so  $\mathfrak{a}\cap\mathfrak{h}$ is generated by the vector $H_{e_{1}-e_{n+1}}$ and all  vectors $H_{e_{i}-e_{j}}$, where $i<j\in \lbrace 1,...,n+1\rbrace \backslash \lbrace 1, n+1\rbrace$. The uniqueness property in Proposition \ref{proposition5} along with the existence of such Example (as in Example \ref{e3}) give us:
\begin{cor}
If $\mathfrak{a}\nsubseteq \mathfrak{h}$ and $\mathfrak{g}_{+}\subseteq \mathfrak{h}$ then $G=\SL(n,\C)$ and $M=M_{2}$. In particular, $M$ is conformally flat.
\end{cor}
\section{Case of parabolic isotropy} \label{Parabolic_Case}
In this part  we assume that the  Borel sub-algebra $\mathfrak{b}=\mathfrak{a}\oplus\mathfrak{g}_{+}$ is  contained in $\mathfrak{h}$. In this case, by \cite[Theorem 1.4]{CharlesMelnick}, $M$ is conformally flat (See  \cite[Proposition 3.3]{BDRZ1}). 

There is a sub-algebra $\mathfrak{l}$ of $\mathfrak{g}_{-}$ such that $\mathfrak{h}=\mathfrak{l}\oplus\mathfrak{a}\oplus\mathfrak{g}_{+}$. One can describe more precisely  the sub-algebra $\mathfrak{l}$. Indeed, since the root spaces are $1-$dimensional, $\mathfrak{g}_{+}\subset \mathfrak{h}$,  there is a subset $\Delta'$ of positive roots of $\Delta$ such that $\mathfrak{l}=\bigoplus_{\beta\in -\Delta'} \mathfrak{g}_{\beta}$ (see \cite[Section~5.7]{K}). Let  $\Pi$ be the standard basis of the canonical root system $\Delta$. By  \cite[Proposition~5.90]{K}), there is a subset $\Pi'$ of $\Pi$ such that $\Delta'=\operatorname{span}(\Pi')$.

\subsection{Maximality of the isotropy sub-algebra}
\begin{defi}
The parabolic sub-algebra $\mathfrak{h}$ is said to be maximal if $\vert\Pi'\vert=\vert\Pi\vert-1$.
\end{defi}

Let $\alpha$ be a simple root of $\Pi$ such that $\mathfrak{g}_{-\alpha}\nsubseteq \mathfrak{h}$ (note that this always exists since $M$ is not trivial). Then $\delta+\alpha$ is also a negative root such that $\mathfrak{g}_{\delta+\alpha}\nsubseteq \mathfrak{h}$. Actually we have more:
\begin{prop}
\label{propprop}
The negative root $\delta+\alpha$ is the minimal root.
\end{prop}
\begin{proof}
Assume by contradiction that there is a positive root $\beta$ such that $\delta+\alpha-\beta$ is  a  negative root. Thus $\mathfrak{g}_{\delta+\alpha-\beta}\nsubseteq \mathfrak{h}$ and hence $\delta-\left( \delta+\alpha-\beta \right)=\beta-\alpha $ is also a negative root which is impossible.
\end{proof}

As a consequence we get:
\begin{cor}
\label{cococoor}
The parabolic sub-algebra $\mathfrak{h}$ is  maximal.
\end{cor} 
\begin{proof}
Assume that there are two simple roots $\alpha_{1},\alpha_{2}\in \Pi\backslash\Pi'$. By Proposition \ref{propprop} both $\delta+\alpha_{1}$ and $\delta+\alpha_{2}$ are minimal roots of $\Delta$. By uniqueness $\delta+\alpha_{1}=\delta+\alpha_{2}$ and hence $\alpha_{1}=\alpha_{2}$.
\end{proof}
\begin{Remark}
Note that so far, we did not impose any restriction on the rank of $\mathfrak{g}$ and thus Corollary \ref{cococoor} remains valid for lower rank semi-simple algebras.
\end{Remark}

\subsection{Higher rank  parabolic case}
We assume that, after modification, the Lie algebra $\mathfrak{g}$ is  of $\operatorname{rank}(\mathfrak{g})\geq 3$. Thus by Proposition \ref{Prop2} it is simple.
\subsubsection{Elimination of cases: first step toward classification}
Let $\alpha$ be the unique simple root of $\Pi\backslash\Pi'$. Then 
using Proposition \ref{propprop} we obtain:
\begin{prop}
\label{Propprosprops}
The simple Lie algebra $\mathfrak{g}$ is of non exceptional type. 
\end{prop}
\begin{proof}
Assume the converse, we now  distinguish several cases depending on the type of $\mathfrak{g}$:
\begin{enumerate}
\item \textbf{If $\mathfrak{g}$ is of type $E_{6}$.} Here  $\delta+\alpha=-\frac{1}{2}\left( e_{8}-e_{7}-e_{6}+e_{5}+e_{4}+e_{3}+e_{2}+e_{1}\right)$. Therefore:
\begin{enumerate}
\item If $\alpha=\alpha_{1}$ then $\delta=-\frac{1}{2}\left( e_{8}-e_{7}-e_{6}+e_{5}+e_{4}+e_{3}+e_{2}+e_{1}\right)-\alpha_{1}$. We have $\mathfrak{g}_{-\left( \alpha_{1}+ e_{2}-e_{1}\right) }\nsubseteq \mathfrak{h}$. However, $\delta+\alpha_{1}+ e_{2}-e_{1}$ is not a root leading to a contradiction.
\item If $\alpha=\alpha_{2}$ then $\delta=-\frac{1}{2}\left( e_{8}-e_{7}-e_{6}+e_{5}+e_{4}+e_{3}+e_{2}+e_{1}\right)-\alpha_{2}$. We have $\mathfrak{g}_{-\left( \alpha_{2}+ e_{3}-e_{2}\right) }\nsubseteq \mathfrak{h}$. But $\delta+\alpha_{2}+ e_{3}-e_{2}$ is not a root leading to a contradiction.
\item If $\alpha=e_{k+1}-e_{k}$ then $\delta=-\frac{1}{2}\left( e_{8}-e_{7}-e_{6}+e_{5}+e_{4}+e_{3}+e_{2}+e_{1}\right)-\left(e_{k+1}-e_{k} \right)$. For $1<k\leq 4$, we have $\mathfrak{g}_{-\left( e_{k+1}-e_{k-1}\right) }\nsubseteq \mathfrak{h}$. However  $\delta+\left( e_{k+1}-e_{k-1} \right) $ leading to a contradiction. For $k=1$, we have $\mathfrak{g}_{-\left( e_{3}-e_{1}\right) }\nsubseteq \mathfrak{h}$. But $\delta+\left( e_{3}-e_{1} \right) $ is also not a root since the coefficient of $e_{2}$ is $-\frac{3}{2}$ so we obtain a contradiction.
\end{enumerate}
\item \textbf{If $\mathfrak{g}$ is of type $E_{7}$.} Here  $\delta+\alpha=-\left( e_{8}-e_{7}\right)$. Thus:
\begin{enumerate}
\item If $\alpha=\alpha_{1}$ then $\delta=-\left( e_{8}-e_{7}\right)-\alpha_{1}$. We have $\mathfrak{g}_{-\left( \alpha_{1}+ e_{3}+e_{2}\right) }\nsubseteq \mathfrak{h}$. But $\delta+\left(\alpha_{1}+ e_{3}+e_{2} \right)=\left(e_{3}+e_{2} \right)-\left(e_{8}-e_{7} \right)$ is not a root leading to a contradiction.
\item If $\alpha=\alpha_{2}$ then $\delta=-\left( e_{8}-e_{7}\right)-\alpha_{2}$. We have $\mathfrak{g}_{-\left( \alpha_{2}+ e_{3}-e_{2}\right) }\nsubseteq \mathfrak{h}$. But $\delta+\left(\alpha_{2}+ e_{3}-e_{2} \right)=\left(e_{3}-e_{2} \right)-\left(e_{8}-e_{7} \right)$ is not a root. So we get a contradiction. 
\item If $\alpha=\alpha_{i}$ with $i>3$ then $\delta=-\left( e_{8}-e_{7}\right)-\alpha_{i}$.  We have $\mathfrak{g}_{-\left( \alpha_{i}+ e_{i-2}+e_{1}\right) }\nsubseteq \mathfrak{h}$. However, $\delta+\left(\alpha_{i}+ e_{i-2}+e_{1} \right)=\left(e_{i-2}+e_{1} \right)-\left(e_{8}-e_{7} \right)$ is not a root leading again to a contradiction.
\item If $\alpha=\alpha_{3}$ then $\delta=-\left( e_{8}-e_{7}\right)-\alpha_{3}$. But $\delta+\left(\alpha_{3}+ e_{3}-e_{2} \right)=\left(e_{3}-e_{2} \right)-\left(e_{8}-e_{7} \right)$ is not a root which contradicts the fact that $\mathfrak{g}_{-\left( \alpha_{3}+ e_{3}-e_{2}\right)}\nsubseteq \mathfrak{h}$.
\end{enumerate}
\item \textbf{If $\mathfrak{g}$ is of type $E_{8}$.} Here $\delta+\alpha=-\frac{1}{2}\left( e_{8}+e_{7}+e_{6}+e_{5}+e_{4}+e_{3}+e_{2}+e_{1}\right)$ and exactly the same proof as for the $E_{6}$ type works.  
\item \textbf{If $\mathfrak{g}$ is of type $F_{4}$.} Here $\delta+\alpha=-\left( e_{1}+e_{2}\right) $. Thus
\begin{enumerate}
\item If $\alpha=\alpha_{1}$ then $\delta=-\alpha_{1}-\left( e_{1}+e_{2}\right)$. We have $\mathfrak{g}_{-\left( \alpha_{1}+ e_{2}+e_{3}\right)}\nsubseteq \mathfrak{h}$. But $\delta+\left(\alpha_{1}+e_{2}+e_{3}\right) $ is not a root leading to a contradiction.
\item If $\alpha=\alpha_{2}$ then $\delta=-\alpha_{2}-\left( e_{1}+e_{2}\right)$. We have $\mathfrak{g}_{-\left( \alpha_{2}+ e_{3}-e_{4}\right)}\nsubseteq \mathfrak{h}$. But $\delta+\left(\alpha_{2}+e_{3}-e_{4}\right) $ is not a root leading to a contradiction.
\item If $\alpha=\alpha_{3}$ then $\delta=-\alpha_{3}-\left( e_{1}+e_{2}\right)$. We have $\mathfrak{g}_{-\left( \alpha_{3}+ e_{4}\right)}\nsubseteq \mathfrak{h}$. However, $\delta+\left(\alpha_{3}+e_{4}\right) $ is not a root leading to a contradiction.
\item If $\alpha=\alpha_{4}$ then $\delta=-\alpha_{4}-\left( e_{1}+e_{2}\right)$. We have $\mathfrak{g}_{-\left( \alpha_{4}+ e_{1}-e_{2}\right)}\nsubseteq \mathfrak{h}$. However, $\delta+\left(\alpha_{4}+e_{1}-e_{2}\right) $ is not a root leading to a contradiction.
\end{enumerate}
\end{enumerate}
\end{proof}
This leads us to the following initial classification of $\mathfrak{g}$:
\begin{prop}
\label{poupoup}
The simple Lie algebra $\mathfrak{g}$ is of type:
\begin{enumerate}
\item $B_{3}$ with $\alpha=e_{3}$ and $\delta=-\left( e_{1}+e_{2}+e_{3}\right) $ or;
\item  $D_{4}$ with $\alpha=e_{3}+e_{4}$ and $\delta=-\left( e_{1}+e_{2}+e_{3}+e_{4}\right)$ or;
\item $D_{4}$ with $\alpha=e_{3}-e_{4}$ and $\delta=-\left( e_{1}+e_{2}+e_{3}-e_{4}\right)$ or;
\item  $B_{n}$ with $n\geq 3$ and $\alpha=e_{1}-e_{2}$ and $\delta=-2e_{1}$ or;
\item $D_{n}$ with $n\geq 3$ and $\alpha=e_{1}-e_{2}$ and $\delta=-2e_{1}$.
\end{enumerate}
\end{prop}
\begin{proof}
For this we distinguish several cases depending on the type of $\mathfrak{g}$. By Proposition \ref{Propprosprops} it is sufficient to consider the non exceptional types: 
\begin{enumerate}
\item \textbf{If $\mathfrak{g}$ is of type $B_{n}$.} Here $\delta+\alpha=-\left( e_{1}+e_{2}\right)$. Thus:
\begin{enumerate}
\item If $\alpha=e_{k}-e_{k+1}$ with  $k\geq 2$. Since $e_{k}$ is a positive root such that $\mathfrak{g}_{-e_{k}}\nsubseteq \mathfrak{h}$ we would then have  $\delta+e_{k}=-\left( e_{1}+e_{2}-e_{k+1}\right) $ is a negative root which is clearly not true;
\item If $\alpha=e_{1}-e_{2}$. In this case $\delta=-2e_{1}$. 
\item If $\alpha=e_{n}$ with $n>3$. Since $e_{n}+e_{3}$ is a positive root such that $\mathfrak{g}_{-\left(e_{n}+e_{3} \right) }\nsubseteq \mathfrak{h}$ we would then have $\delta+e_{n}+e_{3}=-\left( e_{1}+e_{2}-e_{3}\right) $ is a negative root which is clearly not true;
\item If $n=3$ and $\alpha=e_{3}$. In this case $\delta=-\left( e_{1}+e_{2}+e_{3}\right) $. 
\end{enumerate}
\item \textbf{If $\mathfrak{g}$ is of type $C_{n}$.}  Here  $\delta+\alpha=-2e_{1}$. Thus:
\begin{enumerate}
\item If $\alpha=e_{k}-e_{k+1}$ then  $\mathfrak{g}_{-\left( e_{k}+e_{n}\right) }\nsubseteq \mathfrak{h}$. This implies that $\delta+e_{k}+e_{n}=-\left(2 e_{1}-e_{k+1}-e_{n}\right) $ is a negative root which is clearly not true;
\item If $\alpha=2e_{n}$ then $\mathfrak{g}_{-\left( e_{n-1}+e_{n}\right) }\nsubseteq \mathfrak{h}$. This implies that $\delta+e_{n-1}+e_{n}=-\left(2 e_{1}+e_{n}-e_{n-1}\right) $ is a negative root which is clearly not true.
\end{enumerate}
\item \textbf{If $\mathfrak{g}$ is of type $D_{n}$.} Here again $\delta+\alpha=-\left( e_{1}+e_{2}\right)$. Thus:
\begin{enumerate}
\item If $\alpha=e_{k}-e_{k+1}$ with  $2\leq k\leq n-2$. Then  $\mathfrak{g}_{-\left( e_{k}+e_{n-1}\right) }\nsubseteq \mathfrak{h}$. This implies  that $\delta+e_{k}+e_{n-1}=-\left( e_{1}+e_{2}-e_{k+1}-e_{n-1}\right) $ is a negative root which is clearly not true;
\item If $\alpha=e_{n-1}-e_{n}$ and $n\neq 4$. Then $\mathfrak{g}_{-\left( e_{n-2}-e_{n}\right) }\nsubseteq \mathfrak{h}$. This implies  that $\delta+e_{n-2}-e_{n}=-\left( e_{1}+e_{2}+e_{n-1}-e_{n-2}\right) $ is a negative root which is clearly not true;
\item If $n=4$ and $\alpha=e_{3}-e_{4}$ then $\delta=-\left( e_{1}+e_{2}+e_{3}-e_{4}\right)$.
\item If $\alpha=e_{1}-e_{2}$ then in this case $\delta=-2e_{1}$.
\item If $\alpha=e_{n-1}+e_{n}$ with $n\neq 4$. Then $\delta=-\left( e_{1}+e_{2}+e_{n-1}+e_{n}\right)$. But $\delta+\left(e_{3}+e_{n} \right) =-e_{1}-e_{2}-e_{n-1}+e_{3}$ is not a negative  root 
\item If $n=4$ and $\alpha=e_{3}+e_{4}$ then $\delta=-\left( e_{1}+e_{2}+e_{3}+e_{4}\right)$. 
\end{enumerate}
\item \textbf{If $\mathfrak{g}$ is of type $A_{n}$.} Here $\delta+\alpha=-\left( e_{1}-e_{n+1}\right) $ and $\alpha=e_{k}-e_{k+1}$.\\  
If $n \neq 3$ or $k=1,n$, then either $\mathfrak{g}_{-(e_{k-1}-e_{k+1})} \nsubseteq \mathfrak{h}$ or $\mathfrak{g}_{-(e_{k}-e_{k+2})} \nsubseteq \mathfrak{h}$. However, neither $\delta + (e_{k-1}-e_{k+1}) = -\left( e_{1}-e_{n+1}+e_{k}-e_{k+1}\right) $ nor $\delta + (e_{k-1}-e_{k+1}) = -\left( e_{1}-e_{n+1}+e_{k+1}-e_{k+2}\right) $ are negative roots.\\
If $n=3$ and $k=2$ then $\alpha=e_{2}-e_{3}$ so that $\delta=-\left( e_{2}-e_{3}\right) -\left(e_{1}-e_{4} \right)$. In this case $\mathfrak{g}_{e_{1}-e_{2}}\subset \mathfrak{h}$ and  $\mathfrak{g}_{e_{3}-e_{4}}\subset \mathfrak{h}$. But $A_{3}=D_{3}$ so we are in the last case.
\end{enumerate}
\end{proof}
\subsubsection{Recovering the Einstein space}
\label{Parabolic_Case1}
Using the fact that the nilpotent part of $\mathfrak{h}$ acts isometrically we  show:
\begin{prop}
\label{Propparabolic}
The simple Lie algebra $\mathfrak{g}$ is of type:
\begin{enumerate}
\item   $B_{n}$ with $n\geq 3$, $\alpha=e_{1}-e_{2}$ and $\delta=-2e_{1}$ or;
\item  $D_{n}$ with $n\geq 3$, $\alpha=e_{1}-e_{2}$ and $\delta=-2e_{1}$.
\end{enumerate}
\end{prop}
\begin{proof} Following Proposition \ref{poupoup} all we need to prove is that case $(1)$, $(2)$ and $(3)$ are impossible.\\
\textbf{The $\mathfrak{so}(7,\mathbb{C})$ case.} We assume that $\mathfrak{g}$ is  $\mathfrak{so}(7,\mathbb{C})$. It is a complex simple Lie algebra of type $B_{3}$. Its standard root decomposition is described in  \cite[Pages~127-128]{K}. In particular the root spaces are given by, $\mathfrak{g}_{\alpha}=\mathbb{C}E_{\alpha}$. 

We assume that the sub-algebra $\mathfrak{h}$ is generated by $\mathfrak{a}$, $\mathfrak{g}_{+}$, $\mathfrak{g}_{e_{2}-e_{1}}$ and $\mathfrak{g}_{e_{3}-e_{2}}$. In this case $\mathfrak{g}/\mathfrak{h}\simeq \mathfrak{g}_{-e_{1}}\oplus \mathfrak{g}_{-e_{2}}\oplus\mathfrak{g}_{-e_{3}}\oplus\mathfrak{g}_{-e_{2}-e_{3}}\oplus\mathfrak{g}_{-e_{1}-e_{3}}\oplus\mathfrak{g}_{-e_{1}-e_{2}}$ and $\delta=-e_{1}-e_{2}-e_{3}$. 

On the one hand, using Equation \ref{equationsix} with:
\begin{enumerate}
\item $p_{1}=E_{\left(e_{1}-e_{2} \right) }$, $u_{1}=E_{-e_{1}}$, $v_{1}=E_{-\left( e_{1}+e_{3}\right) }$
\item $p_{2}=E_{\left(e_{2}-e_{3} \right) }$, $u_{2}=E_{-e_{2}}$, $v_{2}=E_{-\left( e_{1}+e_{2}\right) }$
\item $p_{3}=E_{-\left(e_{1}-e_{3} \right) }$, $u_{3}=E_{-e_{3}}$, $v_{3}=E_{-\left( e_{2}+e_{3}\right) }$
\end{enumerate}
gives us:
\begin{enumerate}
\item $\langle u_{1},\operatorname{ad}_{p_{1}}v_{1}\rangle +\langle \operatorname{ad}_{p_{1}}u_{1},v_{1}\rangle=0$
\item $\langle u_{2},\operatorname{ad}_{p_{2}}v_{2}\rangle +\langle \operatorname{ad}_{p_{2}}u_{2},v_{2}\rangle=0$
\item 
$\langle u_{3},\operatorname{ad}_{p_{3}}v_{3}\rangle +\langle \operatorname{ad}_{p_{3}}u_{3},v_{3}\rangle=0$
\end{enumerate}
On the other hand we have:
 $\operatorname{ad}_{p_{1}}u_{1}=-2u_{2}$, 
 $\operatorname{ad}_{p_{1}}v_{1}=-2v_{3}$, 
 $\operatorname{ad}_{p_{2}}u_{2}=-2u_{3}$,
 $\operatorname{ad}_{p_{2}}v_{2}=-2v_{1}$,
 $\operatorname{ad}_{p_{3}}u_{3}=-2u_{1}$, and
 $\operatorname{ad}_{p_{3}}v_{3}=-2v_{2}$. This  leads to $$\langle u_{1},v_{3}\rangle=-\langle u_{2},v_{1}\rangle=\langle u_{3},v_{2}\rangle=-\langle u_{1},v_{3}\rangle $$ and hence  $\langle u_{1},v_{3}\rangle=0$ which contradicts the fact that  $\mathfrak{g}_{-e_{1}}$ is paired with $\mathfrak{g}_{-e_{2}-e_{3}}$.\\
\textbf{The $\mathfrak{so}(8,\mathbb{C})$ case.} We assume that $\mathfrak{g}$ is  $\mathfrak{so}(8,\mathbb{C})$. It is a complex simple Lie algebra of type $D_{4}$. Its standard root decomposition is described in  \cite[Pages~128]{K}. In particular the root spaces are given by, $\mathfrak{g}_{\alpha}=\mathbb{C}E_{\alpha}$. 

We assume that the sub-algebra $\mathfrak{h}$ is generated by $\mathfrak{a}$, $\mathfrak{g}_{+}$, $\mathfrak{g}_{e_{2}-e_{1}}$, $\mathfrak{g}_{e_{3}-e_{2}}$ and $\mathfrak{g}_{e_{4}-e_{3}}$ . In this case, $\mathfrak{g}/\mathfrak{h}\simeq \mathfrak{g}_{-\left( e_{1}+e_{2}\right) }\oplus \mathfrak{g}_{-\left( e_{2}+e_{4}\right) }\oplus\mathfrak{g}_{-\left( e_{1}+e_{4}\right) }\oplus\mathfrak{g}_{-\left( e_{2}+e_{3}\right) }\oplus\mathfrak{g}_{-\left( e_{1}+e_{3}\right) }\oplus\mathfrak{g}_{-\left( e_{3}+e_{4}\right) }$ and $\delta=-e_{1}-e_{2}-e_{3}-e_{4}$. 

Using Equation \ref{equationsix} with:
\begin{enumerate}
\item $p_{1}=E_{-\left(e_{2}-e_{3} \right) }$, $u_{1}=E_{-\left( e_{1}+e_{3}\right) }$, $v_{1}=E_{-\left( e_{3}+e_{4}\right) }$
\item $p_{2}=E_{-\left(e_{1}-e_{2} \right) }$, $u_{2}=E_{-\left( e_{2}+e_{3}\right) }$, $v_{2}=E_{-\left( e_{2}+e_{4}\right) }$
\item $p_{3}=E_{\left(e_{1}-e_{3} \right) }$, $u_{3}=E_{-\left( e_{1}+e_{2}\right) }$, $v_{3}=E_{-\left( e_{1}+e_{4}\right) }$
\end{enumerate}
along with the commutation relations: $\operatorname{ad}_{p_{1}}u_{1}=2u_{3}$, 
 $\operatorname{ad}_{p_{1}}v_{1}=2v_{2}$, 
 $\operatorname{ad}_{p_{2}}u_{2}=2u_{1}$,
 $\operatorname{ad}_{p_{2}}v_{2}=2v_{3}$,
 $\operatorname{ad}_{p_{3}}u_{3}=-2u_{2}$, and
 $\operatorname{ad}_{p_{3}}v_{3}=-2v_{1}$
give us $$\langle u_{1},v_{2}\rangle=-\langle u_{3},v_{1}\rangle=\langle u_{2},v_{3}\rangle=-\langle u_{1},v_{2}\rangle $$ and hence  $\langle u_{1},v_{2}\rangle=0$ which contradicts the fact that  $\mathfrak{g}_{-\left( e_{1}+e_{3}\right)}$ is paired with $\mathfrak{g}_{-\left( e_{2}+e_{4}\right)}$.

To finish, assume that the sub-algebra $\mathfrak{h}$ is generated by $\mathfrak{a}$, $\mathfrak{g}_{+}$, $\mathfrak{g}_{e_{2}-e_{1}}$, $\mathfrak{g}_{e_{3}-e_{2}}$ and $\mathfrak{g}_{-e_{3}-e_{4}}$ . In this case, $\mathfrak{g}/\mathfrak{h}\simeq \mathfrak{g}_{-\left( e_{1}+e_{2}\right) }\oplus \mathfrak{g}_{-\left( e_{1}+e_{3}\right) }\oplus\mathfrak{g}_{-\left( e_{2}+e_{3}\right) }\oplus\mathfrak{g}_{-\left( e_{1}-e_{4}\right) }\oplus\mathfrak{g}_{-\left( e_{2}-e_{4}\right) }\oplus\mathfrak{g}_{-\left( e_{3}-e_{4}\right) }$ and $\delta=-e_{1}-e_{2}-e_{3}+e_{4}$. 

Again we use Equation \ref{equationsix} with:
\begin{enumerate}
\item $p_{1}=E_{-\left(e_{2}-e_{3} \right) }$, $u_{1}=E_{-\left( e_{1}+e_{3}\right) }$, $v_{1}=E_{-\left( e_{3}-e_{4}\right) }$
\item $p_{2}=E_{-\left(e_{1}-e_{2} \right) }$, $u_{2}=E_{-\left( e_{2}+e_{3}\right) }$, $v_{2}=E_{-\left( e_{2}-e_{4}\right) }$
\item $p_{3}=E_{\left(e_{1}-e_{3} \right) }$, $u_{3}=E_{-\left( e_{1}+e_{2}\right) }$, $v_{3}=E_{-\left( e_{1}-e_{4}\right) }$
\end{enumerate}
together with the commutation relations: $\operatorname{ad}_{p_{1}}u_{1}=2u_{3}$, 
 $\operatorname{ad}_{p_{1}}v_{1}=2v_{2}$, 
 $\operatorname{ad}_{p_{2}}u_{2}=2u_{1}$,
 $\operatorname{ad}_{p_{2}}v_{2}=2v_{3}$,
 $\operatorname{ad}_{p_{3}}u_{3}=-2u_{2}$, and
 $\operatorname{ad}_{p_{3}}v_{3}=-2v_{1}$
to get $$\langle u_{1},v_{2}\rangle=-\langle u_{3},v_{1}\rangle=\langle u_{2},v_{3}\rangle=-\langle u_{1},v_{2}\rangle $$ and hence  $\langle u_{1},v_{2}\rangle=0$ which contradicts the fact that  $\mathfrak{g}_{-\left( e_{1}+e_{3}\right)}$ is paired with $\mathfrak{g}_{-\left( e_{2}-e_{4}\right)}$.

\end{proof}
Now this last Proposition together with the fact that  we already have examples of such  types (Example \ref{e1}) give us:
\begin{cor}
If $\mathfrak{a}\oplus\mathfrak{g}_{+}\subseteq \mathfrak{h}$ and $\operatorname{rank}(\mathfrak{g})\geq 3$ then $M$ is conformally flat. Moreover, $G=\SO(n+2,\C)$ and $M=\Ein_{n}(\C)$

\end{cor}

\subsection{Classification theorem: lower rank parabolic case}
In this part we need to deal with the parabolic case where after modification the Lie algebra $\mathfrak{g}$ is of $\operatorname{rank}(\mathfrak{g})\leq 2$. 

If $\operatorname{rank}(\mathfrak{g})=1$ then $M$ is conformally equivalent to $\mathbb{CP}^{1}$. If  $\mathfrak{g}$ is of type $A_{1}\times A_{1}$ then, up to finite cover, $G$ is $\operatorname{SL}(2,\mathbb{C})\times \operatorname{SL}(2,\mathbb{C})$ and $H=P_{1}\times P_{2}$ where $P_{1}, P_{2}$ are borel sub-groups of $G$. Hence $M$ is 
 conformally equivalent to $\mathbb{CP}^{1}\times \mathbb{CP}^{1}$. 
 
Now we are left with $A_{2}$, $B_{2}$ or $G_{2}$ types. We have:
\begin{prop}
\label{lastpro}
The Lie algebra $\mathfrak{g}$ is of type:
\begin{enumerate}

\item $B_{2}$ with $\alpha=e_{1}-e_{2}$ and $\delta=-2e_{1}$ or;
\item $G_{2}$ with $\alpha=e_{1}-e_{2}$ and $\delta=-2(e_{3}-e_{2})$.

\end{enumerate}
\end{prop}
\begin{proof}

Assume first that $\mathfrak{g}$ is of type $A_{2}$. In this case $\delta+\alpha=-\left(e_{1}-e_{3}\right) $ and without loss of generality we can suppose that $\alpha=e_{1}-e_{2}$. As $\mathfrak{g}_{-\left( e_{2}-e_{3}\right) }$ acts isometrically,  we use Equation \ref{equationsix} with $0\neq p\in \mathfrak{g}_{-\left( e_{2}-e_{3}\right) }$, $0\neq u=v\in \mathfrak{g}_{-\left( e_{1}-e_{2}\right)}$ to get $\langle\left[p,u \right], u \rangle=0$. But this contradicts the fact that  $\mathfrak{g}_{-\left( e_{1}-e_{2}\right) }$ is paired with $\mathfrak{g}_{-\left( e_{1}-e_{3}\right) }$. 

In the case where $\mathfrak{g}$ is of type $B_{2}$, $\delta+\alpha=-\left( e_{1}+e_{2}\right) $. If $\alpha=e_{2}$ then $\delta=-e_{1}-2e_{2}$. But $\mathfrak{g}_{e_{1}}\nsubseteq \mathfrak{h}$ thus $\delta+e_{1}=-2e_{2}$ is a negative root which is a clearly false. Thus $\alpha=e_{1}-e_{2}$ and $\delta=-2e_{1}$.

Finally, if $\mathfrak{g}$ is of type $G_{2}$ then $\delta+\alpha
=2e_{3}-e_{1}-e_{2}$. Assume that $\alpha=-2e_{1}+e_{2}+e_{3}$, thus  $\delta=-3\left( e_{3}-e_{1}\right) $. As $\mathfrak{g}_{-\left( e_{3}-e_{2}\right)}\nsubseteq \mathfrak{h}$, we have $\delta+\left( e_{3}-e_{2}\right)=3e_{1}-e_{2}-2e_{3}$ is a negative root which is not true.

\end{proof}
\textbf{End of proof of Theorem \ref{theo1}}.
Assume first that $\mathfrak{g}$ is of type $B_{2}$. Then by Proposition \ref{lastpro}, $\alpha=e_{1}-e_{2}$ and $\delta=-2e_{1}$. As we already have an example of such situation we get that $G=\SO(5,\C)$ and $M=\Ein_{3}(\C)$.

To finish we assume that $\mathfrak{g}$ is of type $G_{2}$ with $\alpha=e_{1}-e_{2}$ and $\delta=-2(e_{3}-e_{2})$. In this case 
 the sub-algebra $\mathfrak{h}$ is generated by $\mathfrak{a}$, $\mathfrak{g}_{+}$, and $\mathfrak{g}_{-\left(-2e_{1}+ e_{2}+e_{3}\right)}$ so that $\mathfrak{g}/\mathfrak{h}\simeq \mathfrak{g}_{-\left( e_{1}-e_{2}\right) }\oplus \mathfrak{g}_{-\left( e_{3}-e_{1}\right) }\oplus\mathfrak{g}_{-\left( e_{3}-e_{2}\right) }\oplus\mathfrak{g}_{-\left( -2e_{2}+e_{1}+e_{3}\right) }\oplus\mathfrak{g}_{-\left( 2e_{3}-e_{1}-e_{2}\right)}$. Recall that the root space decomposition of $\mathfrak{g}$ is given by $\mathfrak{g}_{\alpha}=\mathbb{C}E_{\alpha}$ with in particular the following commutation relations, among others:
\begin{enumerate} 
\item $\left[ E_{-\left( -2e_{1}+e_{2}+e_{3}\right) }, E_{-\left( e_{1}-e_{2}\right) }\right]=-E_{-\left( e_{3}-e_{1}\right) } $;
\item $\left[ E_{-\left( -2e_{1}+e_{2}+e_{3}\right) }, E_{-\left( -2e_{2}+e_{1}+e_{3}\right) }\right]=-E_{-\left( 2e_{3}-e_{1}-e_{2}\right) } $;
\item  $\left[ E_{e_{3}-e_{1}}, E_{-\left( e_{3}-e_{2}\right)}\right]=-2E_{-\left( e_{1}-e_{2}\right) } $;
\item $\left[ E_{e_{3}-e_{1}}, E_{-\left( 2e_{3}-e_{1}-e_{2}\right)}\right]=E_{-\left( e_{3}-e_{2}\right) } $;
\item  $\left[ E_{e_{1}-e_{2}}, E_{-\left( e_{3}-e_{2}\right)}\right]=-2E_{-\left( e_{3}-e_{1}\right) } $;
\item $\left[ E_{e_{1}-e_{2}}, E_{-\left( -2e_{2}+e_{1}+e_{3}\right)}\right]=-E_{-\left( e_{3}-e_{2}\right) } $.
\end{enumerate}

On the one hand, $M$ is identified, as a homogeneous space, to the complex Einstein space.

On the other hand, let $\langle.,.\rangle$ be the complex
bilinear form defined on $\mathfrak{g}/\mathfrak{h}$ by: 
\begin{enumerate}
\item $\mathfrak{g}_{-\left( e_{1}-e_{2}\right) }$ is paired with $\mathfrak{g}_{-\left( 2e_{3}-e_{1}-e_{2}\right)}$, $\mathfrak{g}_{-\left( e_{3}-e_{1}\right) }$ with $\mathfrak{g}_{-\left( -2e_{2}+e_{1}+e_{2}\right)}$, and $\mathfrak{g}_{-\left( e_{3}-e_{2}\right) }$ with itself;
\item $\langle E_{-\left( e_{1}-e_{2}\right) }, E_{-\left( 2e_{3}-e_{1}-e_{2}\right)}\rangle=1$;
\item $\langle E_{-\left( e_{3}-e_{1}\right) }, E_{-\left( -2e_{2}+e_{1}+e_{2}\right)}\rangle=-1$;
\item $\langle E_{-\left( e_{3}-e_{2}\right) }, E_{-\left( e_{3}-e_{2}\right)}\rangle=2$;
\end{enumerate}
Then it is worth nothing to verify that the conformal class of $\langle.,.\rangle$ is uniquely  preserved by $\mathfrak{h}$. Thus $M$ admits a unique conformal holomorphic Riemannian structure invariant under the action of the simple Lie group $G_{2}$. In addition, this conformal structure is flat. Hence $M$ is the Einstein space and $G_{2}$ admits a representation in $\operatorname{SO}(7,\mathbb{C})$ (See also \cite{Agricola}).



\bibliographystyle{abbrv}
 \bibliography{Bibliography}

\begin{thebibliography}{10}

\bibitem{Agricola}
I.~Agricola.
\newblock Old and new on the exceptional group {$G_2$}.
\newblock {\em Notices Amer. Math. Soc.}, 55(8):922--929, 2008.

\bibitem{Gromov}
G.~D. Ambra and M.~Gromov.
\newblock Lectures on transformation groups: geometry and dynamics.
\newblock {\em Surveys in differential geometry}, 1:19--111, 1990.

\bibitem{Bader}
U.~Bader and A.~Nevo.
\newblock Conformal actions of simple {L}ie groups on compact
  pseudo-{R}iemannian manifolds.
\newblock {\em J. Differential Geom.}, 60(3):355--387, 2002.

\bibitem{BDRZ1}
M.~Belraouti, M.~Deffaf, Y.~Raffed, and A.~Zeghib.
\newblock Pseudo-conformal actions of the mobius group.
\newblock {\em Differential Geometry and its Applications}, 91:102070, 2023.

\bibitem{Sorindumi}
I.~Biswas and S.~Dumitrescu.
\newblock Holomorphic {R}iemannian metric and the fundamental group.
\newblock {\em Bull. Soc. Math. France}, 147(3):455--468, 2019.

\bibitem{borel}
A.~Borel.
\newblock Groupes lineaires algebriques.
\newblock {\em Annals of Mathematics}, 64(1):20--82, 1956.

\bibitem{ZeghiSorin}
S.~Dumitrescu and A.~Zeghib.
\newblock Global rigidity of holomorphic {R}iemannian metrics on compact
  complex 3-manifolds.
\newblock {\em Math. Ann.}, 345(1):53--81, 2009.

\bibitem{Francesun}
C.~Frances.
\newblock About pseudo-{R}iemannian {L}ichnerowicz conjecture.
\newblock {\em Transform. Groups}, 20(4):1015--1022, 2015.

\bibitem{CharlesMelnick}
C.~Frances and K.~Melnick.
\newblock Formes normales pour les champs conformes pseudo-riemanniens.
\newblock {\em Bull. Soc. Math. France}, 141(3):377--421, 2013.

\bibitem{Francesquatre}
C.~Frances and A.~Zeghib.
\newblock Some remarks on conformal pseudo-{R}iemannian actions of simple {L}ie
  groups.
\newblock {\em Math. Res. Lett.}, 12(1):49--56, 2005.

\bibitem{GasquiGoldschmidt}
J.~Gasqui and H.~Goldschmidt.
\newblock On the geometry of the complex quadric.
\newblock {\em Hokkaido Math. J.}, 20(2):279--312, 1991.

\bibitem{Gasquihuber}
J.~Gasqui and H.~Goldschmidt.
\newblock The infinitesimal rigidity of the complex quadric of dimension four.
\newblock {\em Amer. J. Math.}, 116(3):501--539, 1994.

\bibitem{Kobay}
M.~Inoue, S.~Kobayashi, and T.~Ochiai.
\newblock Holomorphic affine connections on compact complex surfaces.
\newblock {\em J. Fac. Sci. Univ. Tokyo Sect. IA Math.}, 27(2):247--264, 1980.

\bibitem{kups1627}
S.~Klein.
\newblock {\em The complex quadric from the standpoint of Riemannian geometry}.
\newblock PhD thesis, Universit{\"a}t zu K{\"o}ln, 2005.

\bibitem{Klein}
S.~Klein.
\newblock Totally geodesic submanifolds of the complex quadric.
\newblock {\em Differential Geom. Appl.}, 26(1):79--96, 2008.

\bibitem{K}
A.~W. Knapp.
\newblock {\em Lie groups beyond an introduction}, volume 140 of {\em Progress
  in Mathematics}.
\newblock Birkh\"{a}user Boston, Inc., Boston, MA, 1996.

\bibitem{Kobanomizu}
S.~Kobayashi and K.~Nomizu.
\newblock {\em Foundations of differential geometry. {V}ol. {II}}, volume Vol.
  II of {\em Interscience Tracts in Pure and Applied Mathematics, No. 15}.
\newblock Interscience Publishers John Wiley \& Sons, Inc., New
  York-London-Sydney, 1969.

\bibitem{Vincentquatre}
K.~Melnick and V.~Pecastaing.
\newblock The conformal group of a compact simply connected {L}orentzian
  manifold.
\newblock {\em J. Amer. Math. Soc.}, 35(1):81--122, 2022.

\bibitem{Emilio}
E.~Musso and L.~Nicolodi.
\newblock Conformal geometry of isotropic curves in the complex quadric.
\newblock {\em Internat. J. Math.}, 33(8):Paper No. 2250054, 32, 2022.

\bibitem{Vincenttrois}
V.~Pecastaing.
\newblock Essential conformal actions of {$\rm{PSL}(2,\bold{R})$} on
  real-analytic compact {L}orentz manifolds.
\newblock {\em Geom. Dedicata}, 188:171--194, 2017.

\bibitem{Pecastaingun}
V.~Pecastaing.
\newblock Conformal actions of real-rank 1 simple {L}ie groups on
  pseudo-{R}iemannian manifolds.
\newblock {\em Transform. Groups}, 24(4):1213--1239, 2019.

\bibitem{Pecastaingdeux}
V.~Pecastaing.
\newblock Conformal actions of higher rank lattices on compact
  pseudo-{R}iemannian manifolds.
\newblock {\em Geom. Funct. Anal.}, 30(3):955--987, 2020.

\bibitem{Serre}
J.-P. Serre.
\newblock {\em Complex semisimple {L}ie algebras}.
\newblock Springer-Verlag, New York, 1987.
\newblock Translated from the French by G. A. Jones.

\bibitem{Zimmerun}
R.~J. Zimmer.
\newblock Split rank and semisimple automorphism groups of {$G$}-structures.
\newblock {\em J. Differential Geom.}, 26(1):169--173, 1987.

\end{thebibliography}

\end{document}